\documentclass[numbers,webpdf,imaiai]{ima-authoring-template}%

\theoremstyle{thmstyletwo}%
\newtheorem{theorem}{Theorem}
\newtheorem{proposition}[theorem]{Proposition}%

\newtheorem{example}{Example}%
\newtheorem{remark}{Remark}%

\newtheorem{definition}{Definition}
\numberwithin{equation}{section}

\newtheorem{lemma}{Lemma}
\usepackage{cleveref}
\usepackage{natbib}
\setcitestyle{authoryear,round}

\begin{document}

\DOI{}
\copyrightyear{2023}
\vol{}
\pubyear{}
\access{Advance Access Publication Date: Day Month Year}
\appnotes{Paper}
\copyrightstatement{}
\firstpage{1}


\title[Feasible Conjugate Gradient Method]{A Feasible Conjugate Gradient Method for Calculating $\mathcal B$-Eigenpairs of Symmetric Tensors}

\author{Jiefeng Xu
	\address{
		\orgdiv{School of Mathematical Sciences}, \orgname{South China Normal University}, \orgaddress{ \street{Guangzhou}, \state{Guangdong}, \postcode{510631}, \country{People's Republic of China}}
	}
}

\author{Can Li*
	\address{
		\orgdiv{School of Mathematics and Statistics}, \orgname{Honghe University}, \orgaddress{ \street{Mengzi}, \state{Yunnan}, \postcode{661199}, \country{People's Republic of China}}\\
		\orgdiv{School of Mathematical Sciences}, \orgname{South China Normal University}, \orgaddress{ \street{Guangzhou}, \state{Guangdong}, \postcode{510631}, \country{People's Republic of China}}
	}
}

\author{Dong-Hui Li
	\address{
		\orgdiv{School of Mathematical Sciences}, \orgname{South China Normal University}, \orgaddress{ \street{Guangzhou}, \state{Guangdong}, \postcode{510631}, \country{People's Republic of China}}
	}
}


\corresp[*]{Corresponding author: \href{lican@m.scnu.edu.cn}{lican@m.scnu.edu.cn}}

\received{Date}{0}{Year}
\revised{Date}{0}{Year}
\accepted{Date}{0}{Year}


\abstract{In this paper, we propose a feasible conjugate gradient (FCG) method for calculating ${\mathcal B}$-eigenpairs of a symmetric tensor ${\mathcal A}$.
	The method is an extension of the well-known conjugate gradient method for unconstrained optimization problems to some curve constrained optimization problems.
	The proposed FCG method can find a ${\mathcal B}$-eigenpair of a symmetric tensor ${\mathcal A}$ without the requirement that the orders of ${\mathcal A}$ and $\mathcal B$ are equal.
	We pay particular attention to the Polak-Rib\'ire-Polyak (PRP) type conjugate gradient method.
	We show that the FCG method with some Armijo-type line search is globally convergent.
	Our	numerical experiments indicate the promising performance of the proposed method.}


\maketitle

\section{Introduction}
\setcounter{equation}{0}

Let $m$ and $n$ be positive integers and $\mathbb{R}$ be the real field.
An $m$th-order $n$-dimensional real tensor $\mathcal A$ is an array taking the form
\[
{\mathcal A}=(a_{i_1i_2\ldots i_m}),\quad a_{i_1i_2\ldots i_m}\in\mathbb{R}, \ \forall i_{j}\in [n], \ j\in [m],
\]
where $[n]:= \{1,2,\ldots,n\}$.
We use $\mathbb{R}^{[m,n]}$ to denote the set of all real tensors of order $m$ and dimension $n$.
We simply denote $\mathbb{R}^{[1,n]}$ as  $\mathbb{R}^{n}$.
For a tensor $\mathcal{A} \in \mathbb{R}^{[m,n]}$ and a vector $x \in \mathbb{R}^{n}$, we denote the homogenous polynomial
$$
\mathcal{A} x^m:=\sum_{i_{j}\in [n], j\in [m]} a_{i_1 \ldots i_m} x_{i_1} \ldots x_{i_m}.
$$
For a positive integer $k < m$, define $\mathcal{A} x^k$ to be the tensor in $\mathbb{R}^{[m-k,n]}$ with elements
$$
\left(\mathcal{A} x^k\right)_{i_1, \ldots, i_{m-k}}:=\sum_{j_{l}\in [n], l\in [k]} a_{i_1 \ldots i_{m-k} j_1 \ldots j_k} x_{j_1} \ldots x_{j_k} .
$$
Especially, $\mathcal{A} x^{m-1} \in  \mathbb{R}^{n}$ is a vector and ${\mathcal A} x^{m-2}\in \mathbb{R}^{n\times n}$ is a matrix.
The tensor $\mathcal{A}$ is symmetric if each entry $a_{i_1 \ldots i_m}$ is invariant concerning all permutations of $\left(i_1, \ldots, i_m\right)$.
Let $\mathbb{S}^{[m,n]}$ be the set of all symmetric tensors in $\mathbb{R}^{[m,n]}$.
It is well-known that if ${\mathcal A} \in {\mathbb S}^{[m,n]}$, the gradient of the homogenous function ${\mathcal A}x^m$ is $m{\mathcal A}x^{m-1}$
and the Hessian is $m(m-1){\mathcal A}x^{m-2}$.
A tensor $\mathcal{A} \in \mathbb{S}^{[m, n]}$ is called positive definite if
$$
\mathcal{A} x^m>0 \text { for all } x \in \mathbb{R}^n, x \neq 0 .
$$
Let $\mathbb{S}_{+}^{[m, n]}$ denote the space consisting of all positive definite tensors in $\mathbb{S}^{[m, n]}$.

Eigenvalues and eigenvectors of tensors were introduced by \citet{Qi05} and \citet{Lim05}, independently.
They have been found wide applications in magnetic resonance imaging, quantum physics, molecular conformation, independent component analysis, and so on \citep[see e.g.][]{QI08Deigen,Wei03PhysRev,Cardoso99, Lathauwer95HigherorderPM}.
Unlike the matrix case, there are various kinds of eigenvalues and eigenvectors to tensors.
The concept of the so called ${\mathcal B}$-eigenpairs (or generalized eigenpairs) of a symmetric tensor was first given by \citet{CHANG09Eigen}.
A more unified definition of ${\mathcal B}$-eigenpairs below was given by \citet{Cui14AllEig}.
\begin{definition}\label{def:B-eigen}
	\citep{Cui14AllEig} Let ${\mathcal A}\in\mathbb{S}^{[m,n]}$ and $\mathcal{B} \in\mathbb{S}^{[m^{\prime},n]}$ be two symmetric tensors.
	A number $\lambda \in \mathbb{R}$ is a $\mathcal{B}$-eigenvalue of $\mathcal{A}$ if there exists $x \in \mathbb{R}^n \setminus \{0\}$ such that
	\begin{equation}\label{pro:B-eigen}
		\begin{aligned}
			\mathcal{A} x^{m-1}=\lambda \mathcal{B} x^{m^{\prime}-1}, \quad \mathcal{B} x^{m^{\prime}}=1.
		\end{aligned}
	\end{equation}
	Such $x$ is called a $\mathcal{B}$-eigenvector associated with $\lambda$, and such $(x, \lambda)$ is called a $\mathcal{B}$-eigenpair of $\mathcal{A}$.
\end{definition}

Different choices of symmetric tensor ${\mathcal B}$ correspond to different tensor eigenvalue problems.
\begin{remark}
	When $m^{\prime}=m$ is even and $\mathcal{B}$ is the identity tensor (i.e., $\mathcal{B} x^m=x_1^m+\cdots+x_n^m$ ), the $\mathcal{B}$-eigenpair $(x, \lambda)$ is just the $\mathrm{H}$-eigenpair \citep{Qi05,Lim05} satisfying
	$$
	\mathcal{A} x^{m-1}=\lambda x^{[m-1]}, \quad x_1^m+\cdots+x_n^m=1,
	$$
	where $x^{[m-1]} \in \mathbb{R}^{n}$ with elements $\left(x^{[m-1]} \right)_{i} := x_{i}^{m-1}$ for all $i \in [n]$.
\end{remark}
\begin{remark}
	When $m^{\prime}=2$ and $\mathcal{B}$ is the identity tensor (i.e., $\mathcal{B} x^2=x_1^2+\cdots+x_n^2$ ), the $\mathcal{B}$-eigenpair $(x, \lambda)$ is just the Z-eigenpair \citep{Qi05,Lim05} satisfying
	$$
	\mathcal{A} x^{m-1}=\lambda x, \quad x_1^2+\cdots+x_n^2=1.
	$$
\end{remark}
\begin{remark}
	When $m^{\prime}=2$ and $\mathcal{B}$ is the symmetric tensor such that $\mathcal{B} x^2=x^{\top} D x$, where  $D \in \mathbb{R}^{n \times n}$ is a given symmetric positive definite matrix, the $\mathcal{B}$-eigenpair $(x, \lambda)$ is just the D-eigenpair \citep{QI08Deigen} satisfying
	$$
	\mathcal{A} x^{m-1}=\lambda D x, \quad x^{\top} D x=1.
	$$
\end{remark}

There are a lot of works involved in calculating different eigenvalues of tensors.
For calculating Z-eigenvalues of symmetric tensors,
\citet{Qi09Direct} proposed a direct method that can find all Z-eigenvalues of small scale symmetric tensors;
\citet{Kolda11SSHOPM} proposed a shifted power method (SS-HOPM) with convergent guarantee, which is a generalization of the symmetric higher order power method \citep{Kofidis02};
\citet{Hao15TR, Hao15SSPrjM} proposed a feasible trust region algorithm  and a sequential subspace projection method;
\citet{Hu13SSPM} proposed a sequential semidefinite programming method for finding the extreme Z-eigenvalues of even order symmetric tensors.
\citet{Jaffe18NCM} presented a Newton correction method with a locally quadratic convergence;
\citet{Zhao20LCCA} proposed a modified normalized Newton method with  a local and cubical convergence;
\citet{Xu23FNM} proposed a feasible Newton method enjoying a globally quadratic convergence.
Besides, for Z-eigenvalue problems without the symmetric condition of tensors, 
\citet{Benson19Dynamic} presented a dynamical system method and 
\citet{Cui22RGNN} proposed a Rayleigh quotient-gradient neural network model.

When the order $m = m^{\prime}$ is even, the ${\mathcal B}$-eigenvalue problem can be reformulated to a spherical constrained optimization problem \citep{Chen16Comp}.
Based on this reformulation, there are extensive classical optimization methods have been introduced to the  ${\mathcal B}$-eigenvalue problem, such as adaptive shifted power method \citep{Kolda14Adaptive},
inexact steepest descent method \citep{Chen16Comp}, limited memory BFGS algorithm \citep{Chang16SpaHyp}, conjugate gradient methods \citep{Liu19CGLS,Wen22,Zhang23grad},
adaptive cubic regularization method \citep{Chang23ACR}.
For the ${\mathcal B}$-eigenvalue problem (\ref{pro:B-eigen}) of general symmetric tensors, 
\citet{Cui14AllEig} proposed a Jacobian semidefinite relaxation method, which can compute all real ${\mathcal B}$-eigenvalues; 
\citet{Cao19ASTR, Cao20SMBFGS} proposed a feasible self-adaptive trust region method and a subspace modified BFGS method.
What's more, to find all $\mathcal B$-eigenvalues of general tensors, 
\citet{Chen16HM} proposed two homotopy continuation type algorithms, which is effective only for small dimension.
They further proposed a linear homotopy method \citep{Chen17LHM}, which is easy to implement and more effective.
Besides, there also have many algorithms for computing nonnegative eigenvalues of tensors, such as \citet{Ng09NQZ,Ni15QCA,Yang18CubM,Kuo18CMN,Guo19MNI} and the references therein.

In this paper, inspired by the feasible methods proposed by \citet{Hao15TR,Cao19ASTR} and \citet{Xu23FNM}, we propose a feasible conjugate gradient (FCG) method for solving the ${\mathcal B}$-eigenvalue problem (\ref{pro:B-eigen}).
We first reformulate the ${\mathcal B}$-eigenvalue problem to an equality constrained optimization problem.
Then a feasible descent framework is derived by introducing a curve search technique.
The key point of the feasible descent framework is to determine a proper feasible descent direction.
The proposed FCG method determines the feasible descent direction in a way similar to the modified Polak-Rib\'ire-Polyak (PRP) conjugate gradient method \citep{Zhang06MPRP}
for unconstrained optimization problems.
The direction generated by the FCG method is always a feasible descent direction on the feasible curve.
To improve the efficiency of the curve search, we give a simple but useful initial step-length estimate.
Both sequences of the eigenvector estimates and eigenvalue estimates generated by the FCG method are feasible and non-increasing.
We show that the FCG method is globally convergent without any extra requirements.

The rest of this article is organized as follows.
In the next section, we present the feasible conjugate gradient method.
We establish the global convergence of the FCG method in \Cref{sec: conv ana}.
Numerical experiments are presented in \Cref{sec: numer}.
Conclusions are given in the last section.

\section{A Feasible Conjugate Gradient Method}
\setcounter{equation}{0}

\label{sec: fcg}
Consider the following constrained optimization problem
\begin{equation}\label{opt:B-eig}
	\begin{aligned}
		\mathop{\min}_{x\in \mathbb{R}^{n}} f(x): = \frac{1}{m}\mathcal{A} x^{m} \quad \mbox{s.t.} \quad \mathcal{B} x^{m^{\prime}}=1,
	\end{aligned}
\end{equation}
where $\mathcal{A}\in \mathbb{S}^{[m, n]}$, $\mathcal{B} \in \mathbb{S}_{+}^{[m^{\prime}, n]}$ and $m^\prime$ is even.
We denote the feasible set of the problem (\ref{opt:B-eig}) to be
\[
\mathbb{B}:=\{x\in \mathbb{R}^{n}\mid \mathcal{B} x^{m^{\prime}}=1\}.
\]
Since the symmetric tensor $\mathcal{B}$ is positive definite, the feasible set ${\mathbb B}$ is bounded and
the problem (\ref{opt:B-eig}) always has a solution.
Hence $(x,\lambda)$ with $\lambda = \mathcal{A} x^{m}$ is a $\mathcal{B}$-eigenpair of ${\mathcal A}$ if and only if $x$ is a
KKT point of the problem (\ref{opt:B-eig}) and $\lambda$ is its corresponding Lagrangian multiplier.

For a given $(x,d)\in \mathbb B\times \mathbb{R}^n$, we define  a mapping $x(\alpha): \mathbb {R}\to \mathbb {R}^n$ as follows
\begin{equation}
	\label{def:x(alpha d)}
	x (\alpha) = \frac{x + \alpha d}{\|x + \alpha d\|_{\mathcal B}},
\end{equation}
where, for $y\in\mathbb {R}^n$, $\|y\|_{\mathcal B} := \sqrt[m^{\prime}]{{\mathcal B}y^{m^{\prime}}}$.
Throughout this paper, without specification, $\|\cdot\|$ stands for the $l_2$-norm of vectors in $\mathbb{R}^{n}$.
It is easy to see  that $x (\alpha) \in \mathbb{B}$, $\forall \alpha \in \mathbb{R}, x\in\mathbb {B},  d\in \mathbb {R}^n$
as long as $x + \alpha d \neq 0$.
By a simple calculation, we can obtain
\begin{equation}\label{eq: first oder x(alpha)}
	x (\alpha) = x + \alpha\left(I - x \cdot({\mathcal B}x^{m^\prime-1})^{\top}\right)d + o(\|\alpha d\|),\quad
	\forall \alpha \in \mathbb{R}, x\in\mathbb {B},  d\in \mathbb {R}^n
\end{equation}
and
$$\left(I - x\cdot ({\mathcal B}x^{m^{\prime}-1})^{\top}\right)^2 = I - x\cdot ({\mathcal B}x^{m^{\prime}-1})^{\top},\quad
\forall x\in\mathbb {B}.$$
In other words, the matrix $ I - x ({\mathcal B}x^{m^{\prime}-1})^{\top}$ is idempotent.
The equality (\ref{eq: first oder x(alpha)}) implies
\begin{equation}\label{eq:x alpha bdd}
	\begin{split}
		\|x(\alpha) - x\| &\le  \left\| \alpha\left(I - x \left({\mathcal B}x^{m^{\prime}-1}\right)^{\top}\right)d\right\| + o(\|\alpha d\|)\\
		& \le \|\alpha d\| + o(\|\alpha d\|).
	\end{split}
\end{equation}
We define for $x\in \mathbb{R}^{n}$ that
\begin{equation}\label{def:h(x)}
	h(x) := \frac{1}{m} {\mathcal A}\left(\frac{x}{\|x\|_{\mathcal B}}  \right)^{m} = \frac{1}{m}   \frac{{\mathcal A}x^{m} }{\|x\|_{\mathcal B}^{m}}
\end{equation}
and  $\phi(x) := {\mathcal B}x^{m^{\prime}}$.
By direct calculation, the gradient of function $h(x)$ is
\begin{equation}\label{grad h(x)}
	\begin{split}
		\nabla h(x) &=  \|x\|_{\mathcal B}^{-m} \left( {{\mathcal A}x^{m-1} - \phi(x)^{- 1} {\mathcal A}x^{m}  {\mathcal B}x^{m^\prime - 1} } \right)
	\end{split}
\end{equation}
and its Hessian satisfies that
\begin{equation*}
	\begin{split}
		\|x\|_{\mathcal B}^{m} \cdot \nabla^{2} h({x}) & = (m-1)  {\mathcal A}x^{m-2} - (m^\prime -1) \phi(x)^{-1} {\mathcal A}x^{m} {\mathcal B}x^{m^\prime - 2} \\
		&\quad  -m \phi(x)^{-1} \left( {{\mathcal A}x^{m-1} - \phi(x)^{- 1} {\mathcal A}x^{m}  {\mathcal B}x^{m^\prime - 1} } \right) \odot {\mathcal B}x^{m^\prime - 1} \\
		& \quad - (m - m^\prime)  \phi(x)^{-2} {\mathcal A}x^{m} {\mathcal B}x^{m^\prime - 1} \left({\mathcal B}x^{m^\prime - 1}\right)^{\top},
	\end{split}
\end{equation*}
where ${x} \odot {y} \equiv {x} {y}^{\top}+{y} {x}^{\top}$.
If $x \in \mathbb{B}$, the gradient and Hessian of $h(x)$ can be simplified respectively as
\begin{equation}
	\label{F(x) x in B}
	F(x) = {\mathcal A}x^{m-1} - {\mathcal A}x^{m} {\mathcal B}x^{m^{\prime}-1}
\end{equation}
and
\begin{equation}\label{H(x)}
	\begin{split}
		H(x) &= (m-1){\mathcal A}x^{m-2} - (m^\prime -1){\mathcal A}x^{m} {\mathcal B}x^{m^\prime - 2}  \\
		& -m F(x) \odot {\mathcal B}x^{m^\prime-1} - (m - m^\prime) {\mathcal A}x^{m}  {\mathcal B}x^{m^\prime - 1} \left({\mathcal B}x^{m^\prime - 1}\right)^{\top}.
	\end{split}
\end{equation}
For a given feasible point $x\in \mathbb{B}$, we have $f(x(\alpha)) = h(x+\alpha d)$ and $f(x) = h(x)$ by (\ref{def:x(alpha d)}) and (\ref{def:h(x)}).
In this case, we can get by the mean value theorem
\begin{equation}\label{eq: f second}
	f(x(\alpha))  = f(x) + \alpha F(x)^{\top}d + \frac{1}{2} \alpha^2 d^{\top} H(x) d + o(\alpha^2 \|d\|^2).
\end{equation}

The following proposition can be easily obtained from (\ref{F(x) x in B}) and the last equation.
\begin{proposition}\label{prp: desc}
	Let $x \in \mathbb{B}$ and $d\in \mathbb{R}^n$, and $F$ be defined by (\ref{F(x) x in B}).
	The following conclusions are true.
	\begin{itemize}
		\item $F(x) = 0$ if and only if $x$ is a critical point of the problem (\ref{opt:B-eig}), or equivalently $(x, \lambda)$ is a ${\mathcal B}$-eigenpair of ${\mathcal A}$ with $\lambda = {\mathcal A}x^{m}$ .
		\item If the direction $d$ satisfies $F(x)^\top d<0$, then for any constant $\sigma \in(0,1)$, the inequality
		$$
		f(x(\alpha)) \leq f(x)+\alpha \sigma F(x)^{\top}d
		$$
		holds for all $\alpha>0$ sufficiently small.
		In this case, $d$ is called a feasible descent direction of $f$ at $x$.
	\end{itemize}
	Typically, $d=-F(x)$ is a feasible descent direction of $f$ at $x$.
\end{proposition}

The following feasible descent framework provides a way to generate a feasible eigenvector sequence such that its corresponding eigenvalue sequence is non-increasing.
\begin{algorithm}[!htbp]
	\caption{(A feasible descent framework).}\label{alg fdm}
	\begin{algorithmic}[1]
		\State Given an initial point $x_0\in \mathbb{B}$. Let $k:=0$, $\lambda_0={\mathcal A}x_0^m$.
		\While {$F(x_k)\neq 0$}
		\State Determine a direction $d_{k}$ such that $d_{k}^{\top}F(x_{k}) < 0$.
		\State Let $x_k(\alpha)$ be defined by (\ref{def:x(alpha d)}) with $x=x_k$ and $d=d_k$.
		\State Find $\alpha_{k} \in (0, 1]$ such that $f(x_{k}(\alpha_{k})) < f(x_{k})$.
		\State Let $x_{k+1}=x_k(\alpha_k)$ and $\lambda_{k+1}={\mathcal A}x_{k+1}^m$. Let $k:=k+1$.
		\EndWhile
	\end{algorithmic}
\end{algorithm}

In the following, we derive a feasible conjugate gradient method, which we simplify as the FCG method. It can be regarded
as an extension of the modified PRP CG method \citep{Zhang06MPRP} for unconstrained optimization problems.

Given $x_{0}\in \mathbb{B}$.
At iteration $k\ge 0$, suppose that we have already had a feasible point $x_k\in \mathbb B$.
We define the curve  search direction to be
\begin{equation}\label{eq: fcg}
	d_{k} = \left\{\begin{array}{ll}
		-F(x_0), & \text{if}\ k=0, \\
		- F(x_{k}) + \beta_{k-1} d_{k-1} - \theta_{k-1} y_{k-1}, & \text{if}\ k>0,
	\end{array}\right.
\end{equation}
where $y_{k-1} = F(x_{k}) -  F(x_{k-1})$,
\begin{equation}\label{beta,theta}
	\beta_{k-1}= \frac{ F\left(x_k\right)^{\top} y_{k-1} }{\left\|F\left(x_{k-1}\right)\right\|^2}
	\quad \text{and} \quad
	\theta_{k-1} = \frac{F(x_{k})^{\top}d_{k-1}  }{\left\|F\left(x_{k-1}\right)\right\|^2}.
\end{equation}
It follows from (\ref{eq: fcg}) and (\ref{beta,theta}) that
\begin{equation}\label{eq:Fd}
	d_{k}^\top F(x_{k}) = - \|F(x_{k})\|^2\le 0.
\end{equation}
This implies that if the feasible point $x_{k}$ is not a critical point of the problem (\ref{opt:B-eig}), then $d_{k}$ determined by (\ref{eq: fcg})
provides a feasible descent direction of $f$ at $x_k$.
Therefore, we can employ an Armijio-type line search technique to determine $\alpha_{k}>0$  satisfying
\begin{equation}\label{armijio}
	f(x_{k}(\alpha_k )) \le f(x_{k}) +  \sigma_{1} \alpha_k F(x_{k})^{\top}d_{k} -\sigma_{2} \alpha_k^2\|d_{k}\|^2,
\end{equation}
where $\sigma_{1}\in (0, 1)$ and $\sigma_{2}>0$ are given constants.
And let the next iteration be $x_{k+1} = x_{k}(\alpha_{k})$.

We summarize the above process in \Cref{alg1} below.

\begin{algorithm}[!htbp]
	\caption{(A feasible conjugate gradient (FCG) method).}\label{alg1}
	\begin{algorithmic}[1]
		\State Given constants $\rho\in (0,1)$, $\sigma_{1},\sigma_2>0$ and $\delta>0$. Given an initial point $u_0\in \mathbb{R}^{n}\setminus \{0\}$.
		\State Let $k:=0$, $x_{0} = u_{0}/\|u_{0}\|_{{\mathcal B}}$, $\lambda_0={\mathcal A}x_0^m$ and $d_{0} = -F(x_{0})$.
		\While{$F(x_k)\neq 0$}
		\State Compute  $d_k$  by (\ref{eq: fcg}) and  $x_k(\alpha)$  by (\ref{def:x(alpha d)}) with $x=x_k$ and $d=d_k$, respectively.
		\State Determine the steplength $\alpha_k= \max\{\delta \rho^{i}\mid i=0, 1,2, \ldots\}$ satisfying (\ref{armijio}).
		\State Let $x_{k+1}= x_{k}(\alpha_{k})$ and $\lambda_{k+1}={\mathcal A}x_{k+1}^m$. Let $k:=k+1$.
		\EndWhile
	\end{algorithmic}
\end{algorithm}

The FCG method enjoys many nice properties similar to the modified PRP method  \citep{Zhang06MPRP} for unconstrained optimization problems.

\begin{remark}
	The generated eigenvector sequence $\{x_{k}\}$ is feasible and the eigenvalue sequence $\{\lambda_{k}\}$ is non-increasing.
\end{remark}

\begin{remark}\label{remark 1}
	It follows from the line search condition (\ref{armijio}) that the function value sequence $\{f(x_k)\}$ is decreasing and
	$$
	\sum\limits_{k=0}^\infty\alpha_k^2\|d_k\|^2 < \infty,
	$$
	due to the boundedness of $f(x)$ on $x \in \mathbb{B}$.
	It particularly implies $\lim_{k\rightarrow \infty} \alpha_k d_{k} =0$.
\end{remark}

\begin{remark}
	\label{rm 2} Instead of step 5, we can also consider using an exact line search in \Cref{alg1} to get an exact steplength. It means that the steplength
	$\alpha_k=\alpha_{k}^{*}$ is the exact solution of the following one dimensional optimization problem
	\begin{equation}	\label{exc-sub-pro}
		\min_{\alpha\in \mathbb{R}} \phi_{k}(\alpha) := f(x_{k}(\alpha)) = h(x_{k}+\alpha d_{k}).
	\end{equation}
	Denote $x_{k+1}^*=x_{k}(\alpha_k^*)$. It is easy to get   from (\ref{grad h(x)}) that
	\begin{equation}\label{eq Fd=0}
		d_{k}^{\top} F(x_{k+1}^*) = d_{k}^{\top}\nabla h\left(\frac{x_k+\alpha_k^* d_{k}}{\|x_k+\alpha_k^* d_{k}\|_{\mathcal B}}\right)
		=\|x_k+\alpha_k^* d_{k}\|_{\mathcal B} \cdot d_{k}^{\top} \nabla h( x_k+\alpha_k^* d_{k} ) = 0.
	\end{equation}
	If the exact line search (\ref{exc-sub-pro}) is used, the feasible gradient $F(x_{k+1}^{*})$ of $f$ at $x_{k+1}^*$ is orthogonal to the direction $d_{k}$.
	Hence we can see from (\ref{beta,theta}) that $\theta_{k} = 0$.
	Consequently, the direction $d_{k}$ generated by (\ref{eq: fcg}) reduces to the standard PRP CG direction.
\end{remark}

\begin{remark}\label{rm initial step lenth}
	To improve the efficiency of the line search, a proper initial steplength estimate can be obtained by solving
	\begin{equation}
		\begin{split}
			\min_{\alpha \in \mathbb{R}} & f(x_{k}) + \alpha F(x_k)^{\top}d_{k} + \frac{1}{2}\alpha^2 d_{k}^{\top} H(x_{k}) d_{k}.
		\end{split}
	\end{equation}
	That is to say,  we can replace the constant $\delta$ in step 5 of Algorithm \Cref{alg1} by
	\begin{equation}
		\label{delta_k}
		\delta_{k} = \left| \frac{  F(x_k)^\top d_{k}}{d_{k}^{\top} H(x_{k}) d_{k}} \right|.
	\end{equation}
\end{remark}

\section{Convergence Analysis}
\label{sec: conv ana}
In this section, we prove the global convergence of the FCG method proposed in the last section.

It is clear that the vector valued function $F: \mathbb{R}^{n} \rightarrow \mathbb{R}^{n}$ defined by (\ref{F(x) x in B}) is continuously differentiable on $\mathbb{R}^n$.
Hence $F(x)$ has an upper bound on the closed bounded set $\mathbb{B}$, i.e., there exists a constant $M_{1}>0$ such that
\begin{equation}
	\label{bdd F}
	\|F(x)\| \le M_{1}, \quad \forall x \in \mathbb{B}.
\end{equation}
What's more, $F(x)$ is Lipschitz continuous on $x\in \mathbb{B}$, i.e., there exists a constant $L > 0$ such that
\begin{equation}
	\label{Lip F}
	\|F(x)-F(y)\|  \le L \| x- y\|, \quad \forall x,y \in \mathbb{B}.
\end{equation}

The following lemma indicates that the direction sequence $\{d_{k}\}$ is bounded if $\{\left\|F(x_{k})\right\|\}$ has a positive lower bound.
\begin{lemma}\label{lemma: d bdd}
	Let $\left\{x_k\right\}$ and $\left\{d_k\right\}$ be  generated by  \Cref{alg1}.
	If there exists a constant $\epsilon>0$ such that
	$$
	\left\|F(x_{k})\right\| \ge \epsilon, \quad \forall k\ge 0,
	$$
	then the direction sequence $\{d_{k}\}$ is bounded, i.e., there exists a constant $M_{2}>0$ such that
	\begin{equation}\label{upper bound of d}
		\|d_k\|\leq M_2, \quad \forall k\ge0.
	\end{equation}
\end{lemma}
\begin{proof}
	By the definition of $d_k$, we have from (\ref{Lip F}) and (\ref{bdd F}) that
	\begin{equation}
		\label{ineq dk}
		\|d_k\| \le \|F(x_k)\|+\frac{2\|F(x_k)\|\cdot \|y_{k-1}\|}{\|F(x_{k-1})\|^2}\|d_{k-1}\|
		\leq M_{1}+2\frac{M_{1} L \|x_k-x_{k-1}\| }{\epsilon^2}\|d_{k-1}\|.
	\end{equation}
	By using (\ref{eq:x alpha bdd}), we obtain
	\begin{equation*}
		\begin{split}
			\|x_{k}-x_{k-1}\|  = \|x_{k-1}(\alpha_{k-1})-x_{k-1}\| \le \alpha_{k-1}\|d_{k-1}\| +o(\alpha_{k-1}\|d_{k-1}\|)
		\end{split}
	\end{equation*}
	Since $\lim\limits_{k\rightarrow \infty} \alpha_{k}\|d_{k}\| = 0$ by \Cref{remark 1},
	we have $\lim\limits_{k\rightarrow \infty} \|x_{k}-x_{k-1}\| = 0$. As a result,
	there exists a constant $r\in(0,1)$ and an integer $k_0$ such that the inequality
	$$
	2\frac{M_{1} L \|x_k-x_{k-1}\| }{\epsilon^2} \le r
	$$
	holds for all $k\geq k_0$.
	The last inequation and (\ref{ineq dk}) show that
	\begin{equation*}
		\begin{split}
			\|d_k\| &\leq M_1+r\|d_{k-1}\|\\
			&\leq M_1(1+r+\cdots+r^{k-k_0-1})+r^{k-k_0}\|d_{k_0}\|\\
			& \leq \frac{M_1}{1-r}+\|d_{k_0}\|.
		\end{split}
	\end{equation*}
	Let $M_2=\max\left\{\|d_0\|,\|d_1\|,\cdots,\|d_{k_0}\|,\frac{M_1}{1-r}+\|d_{k_0}\|\right\}$.
	Then we get (\ref{upper bound of d}).
\end{proof}
Now we are ready to establish the following global convergence theorem for \Cref{alg1}.
\begin{theorem}\label{thr: g c}
	Let $\left\{x_k\right\}$ and $\left\{d_k\right\}$ be generated by  \Cref{alg1}. Then we have
	\begin{equation}\label{global convergence}
		\liminf\limits_{k\to \infty}\|F(x_{k})\|=0.
	\end{equation}
\end{theorem}
\begin{proof}
	For the sake of contradiction, we suppose that the conclusion is not true, i.e., there exists a constant $\epsilon >0$ such that
	\begin{equation}\label{ass}
		\|F(x_{k})\|\geq \epsilon,\quad \forall k\ge 0.
	\end{equation}
	
	By \Cref{remark 1}, $\lim\limits_{k\to \infty} \alpha_{k}d_k=0$.
	We can obtain from (\ref{eq:Fd}) that
	$$
	\|F(x_{k})\|^2 = - F(x_{k})^\top d_{k} \le \|F(x_{k})\| \cdot \|d_{k}\|,
	$$
	that is $\|F(x_{k})\| \le \|d_{k}\|$.
	If $\liminf\limits_{k\to \infty}\alpha_k>0$, then $\lim\limits_{k\to \infty} \|F(x_{k})\| \le \lim\limits_{k\to \infty}\|d_k\|=0$, which contradicts to  (\ref{ass}).
	
	Suppose that $\liminf\limits_{k\to \infty}\alpha_k=0$.
	Then there is an infinite index set $K$ such that
	$$
	\lim\limits_{k\in K, k\to \infty}\alpha_k=0.
	$$
	According to the line search rule, when $k\in K$ is sufficiently large, $\rho^{-1}\alpha_k$ will  not satisfy (\ref{armijio}).
	That is,
	\begin{equation}\label{not satf line search}
		f\left(x_k(\rho^{-1}\alpha_k)\right)-f(x_k)>-\sigma_{1} \rho^{-1}\alpha_k \|F(x_k)\|^2-\sigma_{2}\rho^{-2}\alpha_k^{2}\|d_k\|^2.
	\end{equation}
	On the other hand, by the boundedness of $\{x_{k}\} \in \mathbb{B}$, there exists a constant $C>0$ such that
	$$
	\frac{1}{2}d_{k}^{\top}H(x_k)d_{k} \le C \|d_{k}\|^2.
	$$
	Therefore, we have from (\ref{eq: f second}) and (\ref{eq:Fd}) that
	\begin{equation}
		\begin{split}
			f(x_{k}(\rho^{-1}\alpha_{k})) - f(x_{k}) &=   \rho^{-1}\alpha_{k} F(x_k)^\top d_{k} + \frac{1}{2}\rho^{-2}\alpha_{k}^2 d_{k}^{\top}H(x_{k})d_{k} + o(\rho^{-2}\alpha_{k}^2\|d_{k}\|^2)\\
			&\le  -\rho^{-1}\alpha_{k} \|F(x_k)\|^2 + C\rho^{-2}\alpha_{k}^2\|d_{k}\|^2 + o(\alpha_{k}^2\|d_{k}\|^2).
		\end{split}
	\end{equation}
	Combining the last inequality with (\ref{not satf line search}), we get for all $k\in K$ sufficiently large,
	\begin{equation*}
		\|F(x_k)\|^2< \frac{\sigma_{2}+C}{(1-\sigma_{1})\rho}\alpha_k\|d_k\|^2+{o(\alpha_k\|d_k\|^2)}.
	\end{equation*}
	Since $\left\{d_k\right\}$ is bounded and $\lim _{k \in K, k \rightarrow \infty} \alpha_k=0$, the last inequality implies
	$$\lim _{k \in K, k \rightarrow \infty}\left\|F(x_k)\right\|=0.$$
	This also yields a contradiction.
	Hence (\ref{global convergence}) is true.
\end{proof}
The following theorem is directly obtained by the boundedness of $\{x_{k}\}$, Proposition \ref{prp: desc}, and \Cref{thr: g c}.
\begin{theorem}
	Let $\left\{x_k\right\}$ and $\left\{\lambda_k\right\}$ be generated by \Cref{alg1}.
	Then $\{\lambda_k\}$ monotonically converges to a ${\mathcal B}$-eigenvalue $\lambda^*$ of the symmetric tensor $\mathcal A$ and any accumulation point $x^{*}$ of $\{x_{k}\}$ is a ${\mathcal B}$-eigenvector associated to $\lambda^*$.
\end{theorem}

\section{Numerical experiments}
\label{sec: numer}
In this section, numerical experiments are conducted to solve the ${\mathcal B}$-eigenvalue problem (\ref{pro:B-eigen}).
We pay particular attention to the Z-, H- and D-eigenvalue problems.
We compare the proposed FCG method with the GEAP method \citep{Kolda14Adaptive}, which is an easy to implement but promising algorithm.
Comparative results are provided to illustrate the performance of our method.
All numerical computations are conducted in MATLAB (R2022a) on a Huawei desktop with Intel(R) Core(TM) i7-10510U CPU at $1.80 \mathrm{GHz} - 2.30 \mathrm{GHz}$ and 16GB of memory running Windows 11.
Tensor toolbox for Matlab \citep{Kolda06} is employed to process tensor computation.

For all experiments, both the FCG and GEAP methods aim to return the maximum tensor eigenvalues.
In \Cref{alg1}, we set the parameters $\sigma_{1}=\sigma_{2}=10^{-4}$, $\rho=0.1$, and the initial step-length $\delta = \delta_{k}$ as (\ref{delta_k}) for all problems.
The parameter '$\tau$' in GEAP is always set to $\tau = 10^{-6}$.
All compared methods start from the same initial vector with entries randomly selected from $[-1,1]$ and normalized to the feasible set $\mathbb{B}$.
We test $1000$ random initial points for \Cref{ex1} and \Cref{ex6}, and use $100$ random initial points for other examples.
The stopping criterion is set to $\mathrm{Res}\leq 10^{-8}$, where
\begin{equation}\label{eq: Res}
	\mathrm{Res}:=
	\begin{cases}
		\left\|\mathcal{A}x^{m-1} - \lambda  \mathcal{B}x^{m-1} \right\|, & \text{if} \ |\lambda|\le 1\\
		\left\|\frac{\mathcal{A}x^{m-1}}{\lambda} - \mathcal{B}x^{m-1} \right\|, & \text{if} \ |\lambda|>1\\
	\end{cases} \quad \text{with} \quad \lambda = \mathcal{A}x^{m},
\end{equation}
or the number of iterations exceeds $500$.

The tested symmetric tensor ${\mathcal A}$ is given as follows.
\begin{example} \label{ex1}
	\citep{Kolda11SSHOPM} Let $\mathcal{A} \in \mathbb{R}^{[3,3]}$ be a symmetric odd order tensor defined by
	$$
	\begin{array}{lllll}
		a_{111}=-0.1281, & a_{112}=0.0516, & a_{113}=-0.0954, & a_{122}=-0.1958, & a_{123}=-0.1790, \\ a_{133}=-0.2676, & a_{222}=0.3251, & a_{223}=0.2513, &	a_{233}=0.1773, & a_{333}=0.0338.
	\end{array}
	$$
\end{example}
\begin{example} \label{ex6}
	\citep{Kolda11SSHOPM} Let $\mathcal{A} \in \mathbb{S}^{[4,3]}$ be the symmetric tensor defined by
	$$
	\begin{array}{llll}
		a_{1111}=0.2883, & a_{1112}=-0.0031, & a_{1113}=0.1973, & a_{1122}=-0.2485, \\
		a_{1123}=-0.2939, & a_{1133}=0.3847, & a_{1222}=0.2972, & a_{1223}=0.1862, \\
		a_{1233}=0.0919, & a_{1333}=-0.3619, & a_{2222}=0.1241, & a_{2223}=-0.3420, \\
		a_{2233}=0.2127, & a_{2333}=0.2727, & a_{3333}=-0.3054 . &
	\end{array}
	$$
\end{example}
\begin{example}\label{ex2}
	\citep{Hao15TR} Consider the $m$th-order $n$-dimensional symmetric tensor ${\mathcal A} \in \mathbb{S}^{[m,n]}$ with
	$$
	a_{i_1, i_2,\ldots, i_m}=\sin \left(i_1+\cdots+i_m\right), \quad \forall i_1,\ldots,i_m\in[n].
	$$
\end{example}

\begin{example}\label{ex3}
	\citep{Kolda14Adaptive}
	A random symmetric tensor ${\mathcal A} \in \mathbb{S}^{[m,n]}$ is generated as follows: we first randomly select entries from $[-1, 1]$, and then symmetrize the resulting tensor to obtain ${\mathcal A}$.
\end{example}

\begin{example}\label{ex4}
	\citep{Nie14Semidef} Consider the $m$th-order $n$-dimensional symmetric tensor ${\mathcal A} \in \mathbb{S}^{[m,n]}$ with
	$$
	a_{i_1, \ldots, i_m}=\arctan \left((-1)^{i_1} \frac{i_1}{n}\right)+\cdots+\arctan \left((-1)^{i_m} \frac{i_m}{n}\right),\quad \forall i_1,\ldots,i_m\in[n].
	$$
\end{example}

\begin{example}\label{ex5}
	The elements of the symmetric tensor ${\mathcal A}\in\mathbb{S}^{[m,n]}$ are given by
	$$
	a_{i_i\ldots i_m}=v_{i_1}+\cdots+v_{i_m},\quad \forall i_1,\ldots,i_m\in[n],
	$$
	where $v \in [-1, 1]^{n}$ is a random vector with uniform distribution.
\end{example}

\begin{center}
	\begin{table}[b]
		\caption{Numerical results of \Cref{ex1} for calculating Z-eigenpairs} \label{tb:small scale}
		\def\temptablewidth{1\textwidth}
		\begin{tabular*}{\temptablewidth}{@{\extracolsep{\fill}}c|ccc|ccc}\toprule
			\multirow{2}{.15in}{\textsf{$\lambda_{Z}^*$}}  & \multicolumn{3}{c|}{GEAP}  &  \multicolumn{3}{c}{FCG} \\
			& \textsf{occ}  & \textsf{iter} & \textsf{time}  &  \textsf{occ}  & \textsf{iter} & \textsf{time} \\
			\midrule
			-0.0006  &  152 &    17  & 0.00079  &  109 &   8.4  & 0.00048   \\
			0.0180  &  168 &    41  & 0.00198  &  129 &     9  & 0.00056   \\
			0.4306  &  311 &    23  & 0.00113  &  322 &   8.5  & 0.00052   \\
			0.8730  &  369 &    13  & 0.00061  &  440 &   8.3  & 0.00049   \\
			\botrule
		\end{tabular*}
	\end{table}
\end{center}

\begin{center}
	\begin{table}[b]
		\caption{Numerical results of \Cref{ex2} and \Cref{ex3} for calculating Z-eigenpairs} \label{tb2}
		\def\temptablewidth{1\textwidth}
		\begin{tabular*}{\temptablewidth}{@{\extracolsep{\fill}}cc|ccccc|ccccc}\toprule
			
			\multirow{2}{.15in}{m}  & \multirow{2}{.1in}{n}
			& \multicolumn{5}{c|}{GEAP}  &  \multicolumn{5}{c}{FCG} \\
			
			&  & \textsf{largest $\lambda_{Z}^*$}& \textsf{occ}  & \textsf{iter} & \textsf{time} & \textsf{suc}
			& \textsf{largest $\lambda_{Z}^*$}& \textsf{occ}  & \textsf{iter} & \textsf{time} & \textsf{suc}\\
			\midrule
			\multicolumn{12}{c}{\Cref{ex2}}\\
			\midrule
			\multirow{3}{*}{4}
			& 20  & 101.8 &  47 &  71.6  & 0.00982 & 100  & 101.8 &  40 &    13  & 0.00183 & 100 \\
			& 50  & 632.8 &  56 &  72.5  & 0.15687 & 100  & 632.8 &  46 &  14.8  & 0.03553 & 100 \\
			& 80  & 1604 &  51 &  73.9  & 1.35663 & 100  & 1604 &  47 &  15.7  & 0.32963 & 100 \\
			\midrule
			\multirow{3}{*}{5}
			& 10  & 64.85 &  13 &  89.9  & 0.00839 & 100  & 64.85 &  12 &  11.9  & 0.00124 & 100 \\
			& 20  & 359.1 &  21 &  91.5  & 0.13536 & 100  & 359.1 &  24 &  14.2  & 0.02457 & 100 \\
			& 30  & 951.0 &  26 &  91.3  & 1.12273 & 100  & 951.0 &  22 &  14.7  & 0.20897 & 100 \\
			
			\midrule
			\multicolumn{12}{c}{\Cref{ex3}}\\
			\midrule
			\multirow{4}{*}{3}
			& 50  &  6.813 &   7 &   363  & 0.05633 &  37  &  6.813 &  17 &  64.7  & 0.00514 & 100 \\
			&100  &  9.365 &   1 &   447  & 0.26977 &  16  &  9.365 &   5 &  84.1  & 0.02293 & 100 \\
			&150  &  - &   - &     -  & - &   0  & 11.503 &   1 &  99.3  & 0.10218 & 100 \\
			&200  &  - &   - &     -  & - &   0  & 13.343 &   2 &   113  & 0.28354 & 100 \\
			\midrule
			\multirow{3}{*}{4}
			& 20  &  3.958 &   8 &   278  & 0.03050 &  85  &  3.958 &   7 &  45.5  & 0.00438 & 100 \\
			& 50  &  7.252 &   2 &   363  & 0.70717 &  44  &  7.252 &   4 &  65.6  & 0.12526 & 100 \\
			& 80  &  8.742 &   1 &   427  & 6.83964 &  11  &  8.962 &   1 &  88.7  & 1.43900 & 100 \\
			\botrule
		\end{tabular*}
	\end{table}
\end{center}

\begin{center}
	\begin{table}[b]
		\caption{Numerical results of \Cref{ex4} and \Cref{ex5}  with large sizes for calculating Z-eigenpairs} \label{tb3}
		\def\temptablewidth{1\textwidth}
		\begin{tabular*}{\temptablewidth}{@{\extracolsep{\fill}}cc|ccccc|ccccc}\toprule
			\multirow{2}{.15in}{m}   & \multirow{2}{.1in}{n}
			& \multicolumn{5}{c|}{GEAP}  &  \multicolumn{5}{c}{FCG} \\
			&  & \textsf{largest $\lambda_{Z}^*$}& \textsf{occ}  & \textsf{iter} & \textsf{time} & \textsf{suc}
			& \textsf{largest $\lambda_{Z}^*$}& \textsf{occ}  & \textsf{iter} & \textsf{time} & \textsf{suc}\\
			\midrule
			\multicolumn{12}{c}{\Cref{ex4}}\\
			\midrule\multirow{4}{*}{3}
			&$10^5$ & 1.275e+05 &  40 &  26.2  & 0.00612 & 100 & 1.275e+05 &  40 &  15.9  & 0.00382 & 100 \\
			&$10^6$ & 1.273e+06 &  42 &  26.6  & 0.19658 & 100 & 1.273e+06 &  42 &  17.7  & 0.13694 & 100 \\
			&$10^7$ & 1.273e+07 &  31 &  28.1  & 2.21683 & 100 & 1.273e+07 &  31 &  20.1  & 1.65794 & 100 \\
			
			\midrule\multirow{3}{*}{4}
			&$10^4$ & 1.42e+06 &  68 &  27.7  & 0.00202 & 100 & 1.42e+06 &  68 &    21  & 0.00131 & 100 \\
			&$10^5$ & 4.515e+07 &  59 &  26.7  & 0.00825 & 100 & 4.515e+07 &  59 &  22.6  & 0.00720 & 100 \\
			&$10^6$& 1.43e+09 &  65 &  29.1  & 0.20815 & 100 & 1.43e+09 &  65 &  25.9  & 0.19334 & 100 \\
			\midrule
			
			\multicolumn{12}{c}{\Cref{ex5}}\\
			\midrule \multirow{3}{*}{5}
			& $10^3$ & 2.714e+07 &  22 &    26  & 0.00017 & 100 & 2.714e+07 &  20 &    17  & 0.00008 & 100 \\
			&$10^4$& 8.32e+09 &  35 &  28.6  & 0.00179 & 100 & 8.32e+09 &  30 &  20.8  & 0.00110 & 100 \\
			&$10^5$ & 2.628e+12 &  35 &  29.4  & 0.00630 & 100 & 2.628e+12 &  32 &  25.6  & 0.00569 & 100 \\
			\midrule \multirow{3}{*}{6}
			&$10^2$ & 1.112e+06 &  57 &  22.5  & 0.00011 & 100 & 1.112e+06 &  53 &  16.1  & 0.00005 & 100 \\
			&$10^3$& 8.473e+08 &  61 &  27.9  & 0.00021 & 100 & 8.473e+08 &  57 &  20.9  & 0.00011 & 100 \\
			&$10^4$& 8.943e+11 &  64 &  28.5  & 0.00179 & 100 & 8.943e+11 &  56 &  23.3  & 0.00127 & 100 \\       
			\botrule
		\end{tabular*}
	\end{table}
\end{center}
\begin{figure}[b]
	\center{
		{\includegraphics[width=0.4\textwidth]{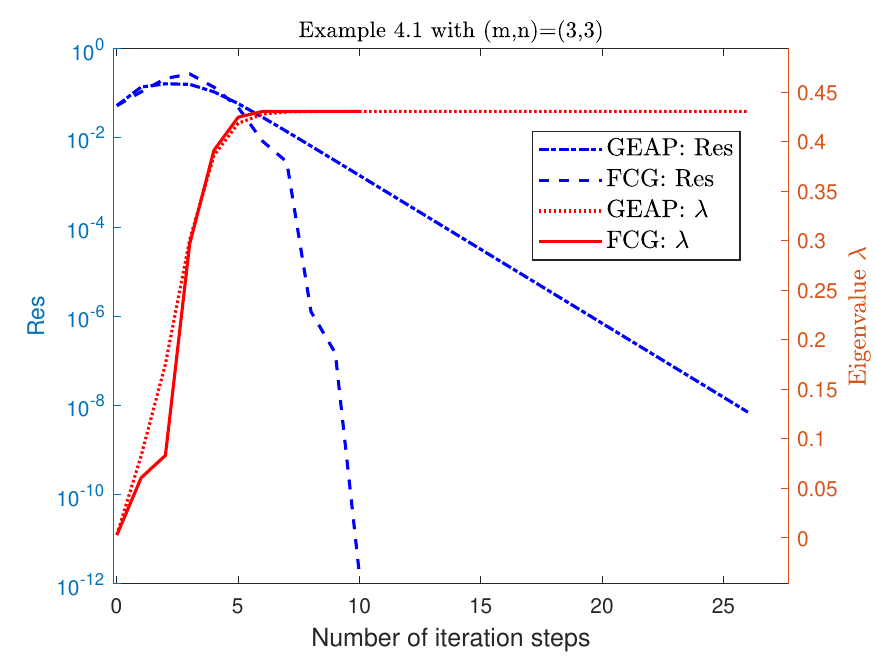}}
		{\includegraphics[width=0.4\textwidth]{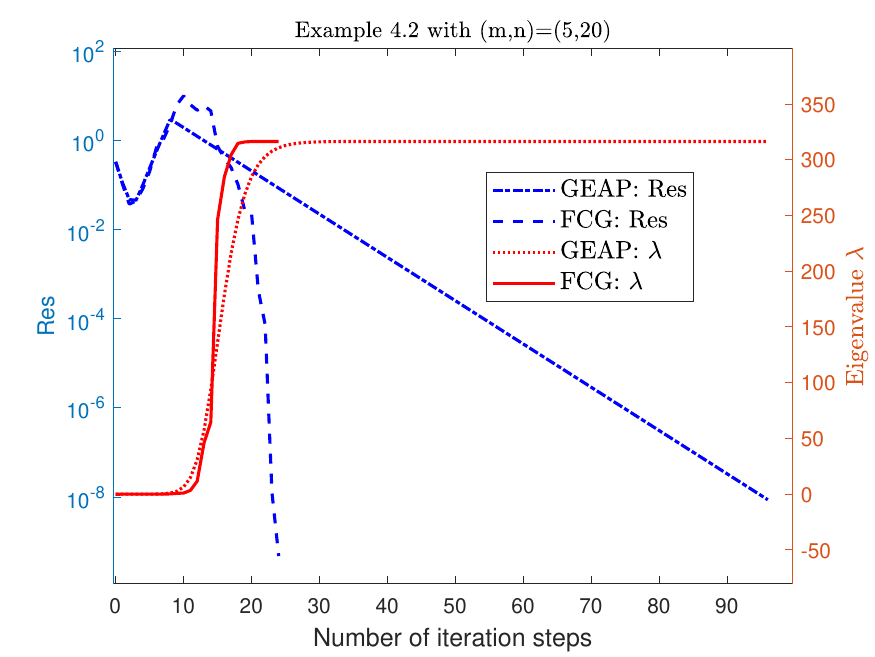}}
	}
	\caption{Evolutions of residue and Z-eigenvalue concerning the number of iteration steps}
	\label{fig1}
\end{figure}

\subsection{Numerical results for computing Z-eigenpairs}
In this subsection, we aim to find Z-eigenpairs of symmetric tensors, i.e.,  $m^{\prime} = 2$ and $\mathcal{B}x^{2} = x^{\top}x$.

We first test the small scale tensor in \Cref{ex1}.
The average numerical results are summarized in \Cref{tb:small scale}, where '$\lambda_{Z}^*$' is the Z-eigenvalue output by each algorithm,
'occ' represents the occurrence number of the corresponding eigenvalue,
'iter' stands for the average iteration number and 'time' means the average CPU time in seconds.

It can be seen from  \Cref{tb:small scale} that FCG and GEAP can always find a Z-eigenvalue at each trial.
The proposed FCG needs fewer iterations and CPU time than GEAP to find a Z-eigenvalue.
We further observe from the Table 3.2 in the literature \citep{Kolda11SSHOPM} that each output Z-eigenvalue in \Cref{tb:small scale} is negative stable, i.e., a local maximum of the problem (\ref{opt:B-eig}).

Now, we consider the symmetric tensors with different orders and dimensions in \Cref{ex2} and \Cref{ex3}.
The average numerical results are reported in \Cref{tb2}, where the column
'largest $\lambda_{Z}^*$' means the largest Z-eigenvalue output by the corresponding algorithms,
and 'suc' refers to the number of successful terminations.

The results of \Cref{ex3} in \Cref{tb2} indicate that both GEAP and FCG found the largest Z-eigenvalue with a low probability.
One possible reason is that the random tensor generated in \Cref{ex3} has many distinct Z-eigenvalues.
Hence, both GEAP and FCG converge easily to a local maximum of the problem (\ref{opt:B-eig}).
It also reveals that as the dimension increases, both GEAP and FCG methods need more iterations and CPU time to terminate.
The success rate of GEAP for finding Z-eigenvalues decreases as the dimension $n$ of the problem increases,
while the FCG method always terminated successfully. 
Moreover, in all successful cases, the FCG method used less CPU time than GEAP.

We then compared the two methods on  \Cref{ex4} and \Cref{ex5} with large sizes. The results are given in \Cref{tb3}.
Besides, to observe the convergence behavior in terms of the number of iterations, \Cref{fig1} presents the visual relations among the residual (\ref{eq: Res}) and eigenvalue estimates with iterations. We see from the figure that  the proposed FCG method
performed much faster than GEAP did, especially in the local region of a Z-eigenpair.

\begin{center}
	\begin{table}[b]
		\caption{Numerical results of \Cref{ex6} for calculating H-eigenpairs} \label{tb sm H}
		\def\temptablewidth{1\textwidth}
		\begin{tabular*}{\temptablewidth}{@{\extracolsep{\fill}}c|ccc|ccc}\toprule
			\multirow{2}{.3in}{\textsf{$\lambda_{H}^*$}}  & \multicolumn{3}{c|}{GEAP}  &  \multicolumn{3}{c}{FCG} \\
			& \textsf{occ}  & \textsf{iter} & \textsf{time}  &  \textsf{occ}  & \textsf{iter} & \textsf{time} \\
			\midrule
			0.8944  &  211 &    38  & 0.00188  &  198 &    10  & 0.00055   \\
			1.9316  &  334 &    42  & 0.00210  &  330 &    11  & 0.00061   \\
			2.3129  &  455 &    46  & 0.00232  &  472 &    11  & 0.00062   \\
			\botrule
		\end{tabular*}
	\end{table}
\end{center}

\begin{center}
	\begin{table}[b]
		\caption{Numerical results of \Cref{ex3} for calculating H-eigenpairs} \label{tb7}
		\def\temptablewidth{1\textwidth}
		\begin{tabular*}{\temptablewidth}{@{\extracolsep{\fill}}cc|ccccc|ccccc}\toprule
			\multirow{2}{.15in}{m}   & \multirow{2}{.1in}{n}
			& \multicolumn{5}{c|}{GEAP}  &  \multicolumn{5}{c}{FCG} \\
			&  & \textsf{largest $\lambda_{H}^*$}& \textsf{occ}  & \textsf{iter} & \textsf{time} & \textsf{suc}
			& \textsf{largest $\lambda_{H}^*$}& \textsf{occ}  & \textsf{iter} & \textsf{time} & \textsf{suc}\\
			\midrule \multirow{3}{*}{4}
			&20 & 54.55 &   2 &   221  & 0.02235 &  91 & 54.55 &   6 &  44.2  & 0.00393 & 100 \\
			&50 & 245.9 &   1 &   291  & 0.61043 &  84 & 249.5 &   1 &    63  & 0.13862 & 100 \\
			&80 & 495.1 &   1 &   330  & 5.38138 &  55 & 494.7 &   1 &  76.2  & 1.38651 & 100 \\
			
			\midrule 
			\multirow{2}{*}{6}
			&10 & 167 &   5 &   168  & 0.05523 &  94 & 167 &   6 &  37.3  & 0.01269 & 100 \\
			&15 & 497.4 &   1 &   180  & 1.09749 &  89 & 497.4 &   1 &  42.9  & 0.28657 & 100 \\
			
			\midrule \multirow{2}{*}{8}
			&5 & 136.5 &  10 &   132  & 0.03041 & 100 & 136.5 &  13 &  26.5  & 0.00611 & 100 \\ 
			&8 & 839.3 &   8 &   194  & 2.91084 &  95 & 839.3 &  10 &  37.9  & 0.61933 & 100 \\ 
			\botrule
		\end{tabular*}
	\end{table}
\end{center}

\subsection{Numerical results for computing H-eigenpairs, D-eigenpairs and general \texorpdfstring{${\mathcal B}$}--eigenpairs }
In this subsection, we first test the performance of FCG for computing H-eigenpairs, i.e., $m^{\prime} = m$ is 
even and $\mathcal{B}x^{m} = \sum_{i=1}^{n}x_{i}^{m}$.
The symmetric tensor ${\mathcal A}$ is generated by \Cref{ex6} and \Cref{ex3}.
The results are listed in Tables \ref{tb sm H} and \ref{tb7}, where $\lambda_{H}^*$ represents the H-eigenvalue 
and the meaning of other columns are the	same as those in Tables \ref{tb:small scale} and \ref{tb2}.

We then compared the methods FCG and GEAP in solving $D$-eigenvalue problems.
In the tested problems,  the symmetric tensor ${\mathcal A}$ was set to the same as that in \Cref{ex3}.
The positive definite symmetric tensor ${\mathcal B}$ was given as follows.
\begin{example} \label{ex7}
	Consider the positive definite symmetric matrix $D\in \mathbb{R}^{n\times n}$ in the form
	$$D=\rho I + CC^{\top},$$
	where $\rho= 0.1$ and $C\in \mathbb{R}^{n\times (n-1)}$ is a matrix whose elements are generated randomly from $[-1, 1]$.
	Let tensor $\mathcal{B} \in \mathbb{S_{+}}^{[m^{\prime},n]}$ with even order $m^{\prime}$ such that ${\mathcal B}x^{m^{\prime}} = \left(x^{\top}Dx\right)^{m^{\prime}/2}$.
	Such ${\mathcal B}$-eigenvalues are D-eigenvalues \citep{QI08Deigen, CHANG09Eigen}.
\end{example}

Tables \ref{tb D even} and \ref{tb D odd} list the performance of both methods, corresponding to the cases where $m$ is even and odd, respectively,
where $\lambda_{D}^*$ represents the D-eigenvalue, and 'iter-in' means the number of inner iterations.
We can observe from \Cref{tb7} and \Cref{tb D even} that the performance of FCG method in calculating the $D$-eigenvalues and 
$H$-eigenvalues of even order tensors is very similar.
The results in \Cref{tb D odd} show that the proposed FCG method also performed quite well in computing $D$-eigenvalues of odd order symmetric tensors.

\begin{center}
	\begin{table}[t]
		\caption{Numerical results for calculating D-eigenpairs of even order symmetric tensor} \label{tb D even}
		\def\temptablewidth{1\textwidth}
		\begin{tabular*}{\temptablewidth}{@{\extracolsep{\fill}}cc|ccccc|ccccc}\toprule
			\multirow{2}{.4in}{($m,m^{\prime}$)} & \multirow{2}{.1in}{n}
			& \multicolumn{5}{c|}{GEAP}  &  \multicolumn{5}{c}{FCG} \\
			&  & \textsf{largest $\lambda_{D}^*$}& \textsf{occ}  & \textsf{iter} & \textsf{time} & \textsf{suc}
			& \textsf{largest $\lambda_{D}^*$}& \textsf{occ}  & \textsf{iter} & \textsf{time} & \textsf{suc}\\
			
			\midrule \multirow{2}{*}{(6,6)}
			&10
			& 0.3072 &  43 &   138  & 0.04743 & 100
			& 0.3072 &  42 &  36.9  & 0.01213 & 100 \\
			&15
			& 0.00717 &  23 &   181  & 1.04812 &  82
			& 0.00717 &  30 &  42.1  & 0.27046 & 100 \\
			\multirow{2}{*}{(8,8)}
			&5
			& 1.745 &  63 &  69.2  & 0.01253 & 100
			& 1.745 &  65 &  26.6  & 0.00476 & 100 \\ 
			&8
			& 0.01916 &  11 &   153  & 2.28275 & 100
			& 0.01916 &   6 &  35.4  & 0.58249 & 100 \\ 
			\botrule
		\end{tabular*}
	\end{table}
\end{center}

\begin{center}
	\begin{table}[t]
		\caption{Numerical results of FCG for calculating D-eigenpairs of odd order symmetric tensor} \label{tb D odd}
		\def\temptablewidth{1\textwidth}
		\begin{tabular*}{\temptablewidth}{@{\extracolsep{\fill}}ccccccccc}\toprule
			($m,m^{\prime}$) & $n$ & \textsf{largest $\lambda_{D}^*$}& \textsf{occ}  & \textsf{iter} & \textsf{iter-in} & \textsf{time} & \textsf{Res} & \textsf{suc}\\
			
			\midrule 
			\multirow{2}{*}{(5,2)}
			&15
			& 0.03288 &  19 &  42.1 &     3  & 0.00923 & 6.5e-09 & 100 \\
			&20
			& 0.0172 &  24 &  47.1 &     4  & 0.07145 & 7.1e-09 & 100 \\
			
			\multirow{2}{*}{(7,2)}
			&5
			& 18.57 &  51 &  31.4 &     1  & 0.00361 & 5.5e-09 & 100 \\
			&10
			& 0.0837 &  11 &  39.4 &     3  & 0.24939 & 6.5e-09 & 100 \\
			\botrule
		\end{tabular*}
	\end{table}
\end{center}

\begin{center}
	\begin{table}[b]
		\caption{Numerical results of FCG for calculating general ${\mathcal B}$-eigenpairs} \label{tb B}
		\def\temptablewidth{1\textwidth}
		\begin{tabular*}{\temptablewidth}{@{\extracolsep{\fill}}ccccccccc}\toprule
			($m,m^{\prime}$) & $n$ & \textsf{largest $\lambda_{{\mathcal B}}^*$}& \textsf{occ}  & \textsf{iter} & \textsf{iter-in} & \textsf{time} & \textsf{Res} & \textsf{suc}\\
			\midrule
			\multirow{2}{*}{(5,4)}
			&10 & 0.0145 &   7 &  31.3 &     2  & 0.00758 & 6.3e-09 & 100 \\
			&20 & 0.0037 &   2 &  41.9 &     4  & 0.08093 & 6.9e-09 & 100 \\
			\multirow{2}{*}{(6,4)}
			&10 & 0.00537 &   4 &  31.1 &     2  & 0.01562 & 6.1e-09 & 100 \\
			&15 & 0.00227 &   3 &    34 &     2  & 0.23490 & 6.4e-09 & 100 \\
			\multirow{2}{*}{(6,6)}
			&10 & 0.00364 &   1 &  35.4 &     2  & 0.02891 & 6.3e-09 & 100 \\
			&15 & 0.00138 &   3 &  40.8 &     2  & 0.55517 & 6.4e-09 & 100 \\
			
			\botrule
		\end{tabular*}
	\end{table}
\end{center}

Finally, we test a randomly generated ${\mathcal B}$-eigenvalue problem.
The symmetric tensor ${\mathcal A}$ is randomly generated in \Cref{ex3}.
The symmetric tensor ${\mathcal B}$ is constructed as an even order diagonally dominant tensor in \Cref{ex8} below, 
which is typically a  positive definite symmetric tensor.	

\begin{example} \label{ex8}
	Let $\mathcal{B} \in \mathbb{S}^{[m^{\prime},n]}$ be a diagonally dominant tensor in the form $\mathcal{B}=s \mathcal{I}+\mathcal{C}$, where ${\mathcal C}\in \mathbb{S}^{[m^{\prime},n]}$ is a tensor whose diagonal entries are zeros and the others generated uniformly from $[-1,1]$, and
	$$
	s=(1+0.01) \cdot \max _{i=1,2, \ldots, n}\left(|\mathcal{C}| \mathbf{e}^{m-1}\right)_i.
	$$
	Here $|\mathcal{C}|$ is the element-wise absolute tensor of ${\mathcal C}$ and $\mathbf{e}=(1,1, \ldots, 1)^{\top} \in \mathbb{R}^{n}$.
\end{example}

The average results are presented in \Cref{tb B}, where $\lambda_{B}^*$ represents the ${\mathcal B}$-eigenvalue.
From \Cref{tb D odd} and \Cref{tb B}, we can see that the number of inner iterations used by the FCG method is very few.
This phenomenon appears very often in all the experiments though we did not list all the inner iterations in other tables.
This reveals that the initial steplength estimate in \Cref{rm initial step lenth} is useful in practice.

\section{Conclusion}
\label{sec: conc}
In this paper, we presented a feasible descent framework for a constrained optimization problem that is equivalent to the ${\mathcal B}$-eigenvalue problem.
Under this framework, we proposed the FCG method by determining the feasible descent direction with the idea of the modified PRP method.
A simple but useful way to set the initial steplength for the FCG method was provided.
The global convergence of the FCG method with an Armijo-type curve search was discussed.
Numerical examples have shown that the method performed quite well in solving different tensor eigenvalue problems.


\section*{Funding}
National Natural Science Foundation of China (grant number 12271187) and the Yunnan Natural Science Foundation (No. 202101BA070001-047).


\end{document}